\documentclass[12pt]{article}
\usepackage{amsmath}
\usepackage{mathtools}
\usepackage{amsfonts}
\usepackage{amsthm}
\usepackage[english]{babel} 
\usepackage[iso-8859-7]{inputenc}
\usepackage{bold-extra}
\usepackage{indentfirst}
\usepackage{hyperref}
\usepackage[margin=1.1in]{geometry}
\newtheorem{definition}{Definition}
\newtheorem{lemma}[definition]{Lemma}
\newtheorem{theorem}[definition]{Theorem}
\newtheoremstyle{r}
{}
{}
{\normalfont}
{}
{\scshape}
{.}
{ }
{}
\theoremstyle{r}
\newtheorem*{remark}{Remark}
\numberwithin{definition}{section}
\title{Universal power series of Seleznev with parameters in several variables}
\author{K. Maronikolakis and G. Stamatiou}
\date{}
\begin{document}
	\maketitle
	\begin{abstract}
		We generalize the universal power series of Seleznev to several variables and we allow the coefficients to depend on parameters. Then, the approximable functions may depend on the same parameters. The universal approximation holds on products $K = \displaystyle\prod_{i = 1}^d K_i$, where $K_i \subseteq \mathbb{C}$ are compact sets and $\mathbb{C} \setminus K_i$ are connected, $i = 1, \dots, d$ and $0 \notin K$. On such $K$ the partial sums approximate uniformly any polynomial. Finally, the partial sums may be replaced by more general expressions. The phenomenon is topologically and algebraically generic.
	\end{abstract}
	\textbf{AMS classification numbers:} 30K05, 32A05.\\
	\textbf{Key words and phrases:} Universal Taylor series, Baire's theorem, generic property, Mergelyan's theorem, product of planar sets.
	\section{Introduction}
	A classical result of Seleznev \cite{Seleznev} states that there exists a formal power series $\sum\limits_{k = 0}^{\infty}a_kz^k,$\break$ a_k \in \mathbb{C}$, such that the partial sums $S_n(z) = \sum\limits_{k = 0}^{n}a_kz^k$ approximate every polynomial uniformly on any compact set $K \subseteq \mathbb{C} \setminus \{0\}$ with $\mathbb{C} \setminus K$ connected. The set of such series is proven to be $G_\delta$ and dense in the set of all power series, identified with the sequence of their coefficients, that is, in $\mathbb{C}^{\aleph_0}$ endowed with the Cartesian topology (\cite{Bayart}). The above set also contains a dense in $\mathbb{C}^{\aleph_0}$ vector space except 0.
	
	If we allow each coefficient $a_k$ to depend on some parameters $w = (w_1, \dots, w_r)$, then the functions $h(z)$ to be approximated can also depend on these parameters $h(w_1, \dots, w_r, z)$ and the approximation is done by the partial sums $S_n(w_1, \dots, w_r, z) = \sum\limits_{k = 0}^{n}a_k(w_1, \dots, w_r)z^k$, $ n = 0, 1,\dots, z \in \mathbb{C}$.
	
	Each one of the parameters $w_i$ varies on a simply connected domain $\Omega_i \subseteq \mathbb{C}, i = 1, \dots, r$. Thus, each $a_k, k = 0, 1, \dots$ is a holomorphic function on $\Omega = \Omega_1 \times \dots \times \Omega_r$. Each power series $\sum\limits_{k = 0}^{\infty}a_k(w_1, \dots, w_r)z^k$ can be identified with the sequence of holomorphic functions $(a_k)_{k = 0}^\infty$ in the space $H(\Omega)^{\aleph_0}$ endowed with the Cartesian topology when the space $H(\Omega)$ is endowed with the topology of uniform convergence on all compact subsets of $\Omega$. This is a Fr\'echet space and Baire's Theorem is at our disposal. The functions to be approximated will be the functions of the form $h(w_1, \dots, w_r, z)$ which can be approximated by polynomials of $r + 1$ variables uniformly on each set of the form $L_1 \times \dots \times L_r \times K$, where $L_i$ is any compact subset of $\Omega_i, i = 1, \dots, r$ and $K \subseteq \mathbb{C} \setminus \{0\}$ is compact and $\mathbb{C} \setminus K$ is connected. Since the sets $\Omega_i$ are simply connected domains in $\mathbb{C}$, without loss of generality we may assume that $\mathbb{C} \setminus L_i$ are connected for all $i = 1, \dots, r$. Furthermore, one can show that there exist power series $\sum\limits_{k = 0}^{\infty}a_k(w_1, \dots, w_r)z^k = (a_k)_{k = 0}^\infty \in H(\Omega)^{\aleph_0}$ with the above properties and their set is a $G_\delta$ and dense subset of $H(\Omega)^{\aleph_0}$ and contains a dense vector space except 0.
	
	But which are the approximable functions $h(w_1, \dots, w_r, z)$? The answer is given in \cite{Falco}, where, for every compact set $L \subseteq \mathbb{C}^m$ a new function algebra $A_D(L)$ is introduced, containing the set $\overline{\mathcal{O}}(L)$ of uniform on $L$ limits of sequences $f_n : V_n \rightarrow \mathbb{C}$ of holomorphic functions in varying open sets $V_n$ such that $L \subseteq V_n \subseteq \mathbb{C}^m$ and where $A_D(L)$ is contained in the classical algebra $A(L) = C(L) \cap H(L^{\mathrm{o}})$, that is $\overline{\mathcal{O}}(L) \subseteq A_D(L) \subseteq A(L)$.  The definition of the new algebra $A_D(L)$ is the following: A function $f : L \rightarrow \mathbb{C}$ belongs to $A_D(L)$, if and only if, $f \in A(L)$ and $f$ is holomorphic on each analytic disc contained in $L$, even meeting the boundary of $L$; that is a function $f : L \rightarrow \mathbb{C}$ belongs to $A_D(L)$, if and only if $f$ is continuous on $L$ and for every disk $D \subseteq \mathbb{C}$ and every injective holomorphic mapping $\phi : D \rightarrow L \subset \mathbb{C}^m$, the composition $f \circ \phi : D \rightarrow \mathbb{C}$ is holomorphic on D.
	
	In \cite{Falco}, it is also proven that, if $L$ is of the form $L = \displaystyle\prod_{i = 1}^m L_i, L_i \subseteq \mathbb{C}$ compact with $\mathbb{C} \setminus L_i$ connected for all $i = 1, \dots, m$, then the set of uniform on $L$ limits of polynomials coincide with $A_D(L)$.
	
	It follows that the approximable functions $h(w_1, \dots, w_r, z)$ are the continuous functions $h : \Omega \times K \rightarrow \mathbb{C}$, such that , for every compact sets $L_i \subseteq \Omega_i$ with $\mathbb{C} \setminus L_i$ connected for all $i = 1, \dots, r$, the restriction of $h$ to the Cartesian product $\big(\displaystyle\prod_{i = 1}^r L_i\big) \times K$ belongs to $A_D\Big[\big(\displaystyle\prod_{i = 1}^r L_i\big) \times K\Big]$. Obviously, such a function $h$ is holomorphic on $\Omega$ when the last variable $z \in K$ is fixed. Also, $h$ is holomorphic on $\Omega \times K^{\mathrm{o}}$ and for $w = (w_1, \dots, w_r) \in \Omega$ fixed, the function $K \ni z \rightarrow h(w_1, \dots, w_r, z) \in \mathbb{C}$ belongs to $A_D(K)$.
	
	Next, we can repeat the previous, replacing the complex variable $z$ by the $d$-tuple of complex variables $(z_1, \dots, z_d)$.
	
	We set $z = (z_1, \dots, z_d), m = (m_1, \dots, m_d) \in \{0, 1, \dots\}^d$ and we denote by $z^m$ the product $z_1^{m_1} \cdot \dots \cdot z_d^{m_d}$. Then, we consider the formal power series $\sum\limits_{m \in \{0, 1, \dots\}^d}a_mz^m$. For such a series, there are several kinds of partial sums; for instance $\sum\limits_{m \in \{0, 1, \dots, n\}^d}a_mz^m$ or $\sum\limits_{m : m_1^2 + \dots + m_d^2 \leq n}a_mz^m$. In order to decide which partial sums we will use for the universal approximation we consider an arbitrary enumeration $N_k, k = 0, 1, \dots$ of $\{0, 1, \dots\}^d$, we fix it and we consider the partial sums $S_n(z) = \sum\limits_{k = 0}^{n}a_{N_k}z^{N_k}, n = 0, 1, \dots$. We also consider an infinite subset $\mu$ of $\{0, 1, \dots\}$. The partial sums that we will use for the universal approximation are $S_n(z), n \in \mu$.
	
	By these partial sums, we will approximate any polynomial uniformly on sets $K$ of the form $K = \displaystyle\prod_{i = 1}^d K_i$ where $K_i$ are compact subsets of $\mathbb{C}$ with $\mathbb{C} \setminus K_i$ connected and $0 \notin K$ (equivalently, there is at least one $i_0 \in \{1, \dots, d\}$ such that $0 \notin K_{i_0}$). Because polynomials are uniformly dense in $A_D(K)$ according to the result of \cite{Falco}, the approximable functions are $h \in A_D(K)$. The set of such universal power series $\sum\limits_{m \in \{0, 1, \dots\}^d}a_mz^m \equiv (a_{N_k})_{k = 0}^\infty \in \mathbb{C}^{\aleph_0}$ is a $G_\delta$ and dense subset of $\mathbb{C}^{\aleph_0}$ endowed with the Cartesian topology and contains a dense vector space except 0. Notice that the above are extensions of the classical theorem of Seleznev, not only in several variables, but also in one variable, because using the arbitrary enumeration $N_k$ we can change the order of the monomials appearing in the power series and use for approximation other partial sums, not the usual ones.
	
	If in the last consideration we allow that the coefficients $a_m, m \in \{0, 1, \dots\}^d$ depend on some parameters $w_1, \dots, w_r$ varying in some simply connected planar domains $\Omega_i, i = 1, \dots, r (w_i \in \Omega_i)$, then the approximable functions may depend on the same parameters as well and be continuous functions $h(w, z)$ depending on $(w, z) = (w_1, \dots, w_r, z_1, \dots, z_d) \in \Omega_1 \times \dots \times \Omega_r \times K_1 \times \dots \times K_d$, that belong to $A_D(L_1 \times \dots \times L_r \times K_1 \times \dots \times K_d)$ for all compact sets $L_i \subseteq \Omega_i, \mathbb{C} \setminus L_i$ connected for $i = 1, \dots, r$. The set of such power series $\sum\limits_{m \in \{0, 1, \dots\}^d}a_m(w)z^m \equiv (a_{N_k})_{k = 0}^\infty \in H(\Omega)^{\aleph_0}$, where $\Omega = \Omega_1 \times \dots \times \Omega_r$, is a $G_\delta$ dense subset of $H(\Omega)^{\aleph_0}$ and contains a dense vector subspace except 0. Here, $H(\Omega)$ is endowed with the topology of uniform convergence on each compact subset of $\Omega$ and $H(\Omega)^{\aleph_0}$ is endowed with the Cartesian topology.
	
	Finally, following \cite{Maronikolakis} we can use for approximation instead of the partial sums $S_n = \sum\limits_{k = 0}^{n}a_{N_k}(w)z^{N_k}$ expressions of the form $\sum\limits_{k = 0}^{n}(b_k(a_0(w), a_1(w) \dots , a_k(w)))z^{N_k}, k \in \mu$, where $b_k : H(\Omega)^{k + 1}\rightarrow H(\Omega)$ are continuous functions such that for every $a_0, a_1, \dots , a_{k - 1} \in H(\Omega)$, the image of the function $H(\Omega)\ni f \rightarrow b_k(a_0, a_1, \dots , a_{k - 1}, f)$ contains all polynomials for all $k = 1, 2, \dots$(If $k = 0$, then we assume that the image of the function $H(\Omega)\ni f \rightarrow b_0(f)$ contains all polynomials). If $U^\mu$ denotes the set of such universal series in $H(\Omega)^{\aleph_0}$, then we will prove that $U^\mu$ is a $G_\delta$ dense subset of $H(\Omega)^{\aleph_0}$. If in addition each $b_k$ is a linear map, then $U^\mu$ contains a vector subspace except 0 dense in $H(\Omega)^{\aleph_0}$.
	
	These two last statements imply all previously mentioned ones. A special case is the Ces\`aro transformation of power series, given by $b_k(a_0, a_1, \dots, a_{k - 1}, a_k) = \frac{a_0 + a_1 + \dots + a_{k - 1} + a_k}{k + 1}$.
	\begin{remark}
		In \cite{Maronikolakis}, there are not parameters, thus $\displaystyle\prod_{i \in \emptyset} \Omega_i$ is a singleton. Then $H(\Omega) = \mathbb{C}$ and the set of polynomials in $H(\Omega)$ is also $\mathbb{C}$. Thus, the assumption that the image of the function $H(\Omega)\ni f \rightarrow b_k(a_0, a_1, \dots , a_{k - 1}, f)$ contains the polynomials reduces to the requirement that the function $\mathbb{C} \ni z \rightarrow b_k(a_0, a_1, \dots , a_{k - 1}, z) \in \mathbb{C}$ is onto, which is the assumption in \cite{Maronikolakis}. The same is true in the extension of \cite{Maronikolakis} in several variables without any parameters ($z = (z_1, \dots, z_d), m = (m_1, \dots, n_d), \sum\limits_{m \in \{0, 1, \dots, n\}^d}a_mz^m$). However, at the end of the paper, we include a direct proof of this result.
	\end{remark}
	The first paper where universal functions depending on parameters are considered is \cite{Abakumov}.
	\section{The Algebra $A_D(L)$}
	In this section we introduce some context from \cite{Falco} which will be useful later on.
	\begin{definition}\cite{Falco}
		Let $L \subset \mathbb{C}^n$ be a compact set. A function $f : L \rightarrow \mathbb{C}$ is said to belong to the class $A_D(L)$ if it is continuous on $L$ and, for every open disk $D \subset \mathbb{C}$ and every injective mapping $\phi : D \rightarrow L \subset \mathbb{C}^n$ holomorphic on $D$, the composition $f \circ \phi : D \rightarrow \mathbb{C}$ is holomorphic on D.
	\end{definition}
	We recall that a function $\phi : D \rightarrow \mathbb{C}^n$, where $D$ is a planar domain, is holomorphic if each coordinate is a holomorphic function.
	
	We have the following approximation lemma, which is part of a result in \cite{Falco}, Remark 4.9 (3).
	\begin{lemma}\cite{Falco}
	\label{Merg}
		Let $L = \displaystyle\prod_{i = 1}^n L_i$ where the sets $L_i$ are compact subsets of $\mathbb{C}$ with connected complement for $i = 1, 2, \dots, n$. If $f \in A_D(L)$ and $\varepsilon > 0$, then there exists a polynomial $p$ such that $\sup\limits_{z \in L}|f(z) - p(z)| < \varepsilon$.
	\end{lemma}
	\section{Main Result}
	\begin{definition}
		Let $\Omega_i \subseteq \mathbb{C}$ be simply connected domains, $i = 1, \dots, r$ and $\Omega = \displaystyle\prod_{i = 1}^{r} \Omega_i$. For every integer $k\geq0$ let $b_k : H(\Omega)^{k + 1}\rightarrow H(\Omega)$ be a continuous function such that for every $a_0, a_1, \dots , a_{k - 1} \in H(\Omega)$, the image of the function $$H(\Omega)\ni f \rightarrow b_k(a_0, a_1, \dots , a_{k - 1}, f)$$ contains the polynomials.
		
		Let $N_k, k = 0, 1, \dots$ be an enumeration of $\mathbb{N}^d$ and $a = (a_0, a_1, \dots) \in H(\Omega)^{\aleph_0}$. For every integer $n \geq 0, w \in \Omega$ and $z \in \mathbb{C}^d$ we set $S_n(a)(w)(z) = \sum\limits_{k = 0}^{n}(b_k(a_0, \dots , a_k))(w)z^{N_k}$. Let $\mu$ be an infinite subset of $\mathbb{N}$. We define $U^\mu$ to be the set of $a \in H(\Omega)^{\aleph_0}$ such that for every set $K = \displaystyle\prod_{i = 1}^{d} K_i$ where $K_i \subseteq \mathbb{C}$ are compact with $\mathbb{C} \setminus K_i$ connected for all $i = 1, \dots d$ and there is at least one $i_0 \in \{1, \dots, d\}$ such that $0 \notin K_{i_0}$ and every continuous function $h : \Omega \times K \rightarrow \mathbb{C}$ such that for every fixed $w \in \Omega$, the function $K \ni z \rightarrow h(w, z) \in \mathbb{C}$ belongs to $A_D(K)$ and for every fixed $z \in K$, the function $\Omega \ni w \rightarrow h(w, z) \in \mathbb{C}$ belongs to $H(\Omega)$, there exists a strictly increasing sequence $\lambda_n \in \mu, n = 1, 2, \dots$ such that $$S_{\lambda_n}(a)(w)(z) \longrightarrow h(w,z)\text{, uniformly on } L \times K \text{, for every compact subset } L \text{ of } \Omega.$$
	\end{definition}
	We notice that if we assume that there exists a sequence of integers $(\lambda_n)_{n \in \mathbb{N}},\lambda_n\in \mu$, not necessarily strictly increasing, such that $(S_{\lambda_n}(a)(w)(z))_{n \in \mathbb{N}}$ converges to $h(w,z)$ uniformly on $L \times K$ for every compact subset $L$ of $\Omega$, then the two definitions are equivalent; see \cite{Vlachou}.
	
	We also have the following equivalence: A continuous function $h : \Omega \times K \rightarrow \mathbb{C}$ has the properties that for every fixed $w \in \Omega$, the function $K \ni z \rightarrow h(w, z) \in \mathbb{C}$ belongs to $A_D(K)$ and for every fixed $z \in K$, the function $\Omega \ni w \rightarrow h(w, z) \in \mathbb{C}$ belongs to $H(\Omega)$ if and only if for every compact sets $L_i \subseteq \Omega_i$ with $\mathbb{C} \setminus L_i$ connected for all $i = 1, \dots, r$, the restriction of $h$ to the Cartesian product $\big(\displaystyle\prod_{i = 1}^r L_i\big) \times K$ belongs to $A_D\Big[\big(\displaystyle\prod_{i = 1}^r L_i\big) \times K\Big]$. This holds, because, as it is proven in \cite{Falco}, a continuous complex function defined on a product $M = \displaystyle\prod_{l = 1}^{\sigma}M_l$ of planar compact sets $M_l$ belongs to $A_D(M)$ if and only if, for every $l \in \{1, \dots, \sigma\}$, the corresponding slice functions belong to $A_D(M_l) = A(M_l).$
	
	Considering the set $U^\mu$  as a subset of the space $H(\Omega)^{\aleph_0}$ endowed with the product topology, we shall prove that $U^\mu$ is a countable intersection of open dense sets. Since $H(\Omega)^{\aleph_0}$ is a metrizable complete space, Baire's Theorem is at our disposal and so $U^\mu$ is a dense $G_\delta$ set.
	
	Let $F_p, p = 1, 2, \dots$ be an exhausting family of compact subsets of $\Omega$, where each $F_p$ is a product of planar compact sets with connected complement. We may also assume that $F_p \subseteq F_{p + 1}$ for all $p = 1, 2, \dots$. It is known that there exist compact sets $R_j \subseteq \mathbb{C} \setminus \{0\}, j = 1, 2, \dots$ with connected complement, such that for every compact set $T \subseteq \mathbb{C} \setminus \{0\}$ with connected complement, there exists an integer $j$ such that $T \subseteq R_j$. Let $T_m$ be an enumeration of all $\displaystyle\prod_{i = 1}^{d} Q_i$, such that there exists an integer $i_0 = 1, 2, \dots, d$ with $Q_{i_0} \in \{R_j : j = 1, 2, \dots,\}$ and the rest of the sets $Q_i$ are closed disks centered at 0 whose radius is a positive integer.
	
	Let $f_j, j = 1, 2, \dots$ be an enumeration of all polynomials of $r + d$ variables having coefficients with rational coordinates. For any integers $p, m, j, s, n$ with $p \geq 1, m \geq 1, j \geq 1, s \geq 1, n \geq 0$, we denote by $E(p, m, j, s, n)$ the set $$E(p, m, j, s, n) := \Big\{a \in H(\Omega)^{\aleph_0} : \sup_{(w,z) \in F_p \times T_m}\big|S_n(a)(w)(z) - f_j(w,z)\big| < \frac{1}{s}\Big\}.$$
	\begin{lemma}
		$U^\mu$ can be written as follows: $$U^\mu = \bigcap_{p=1}^{\infty}\bigcap_{m=1}^{\infty}\bigcap_{j=1}^{\infty}\bigcap_{s=1}^{\infty}\bigcup_{n \in \mu}E(p, m, j, s, n).$$
	\end{lemma}
	\begin{proof}
		The inclusion $U^\mu \subseteq \bigcap\limits_{p=1}^{\infty}\bigcap\limits_{m=1}^{\infty}\bigcap\limits_{j=1}^{\infty}\bigcap\limits_{s=1}^{\infty}\bigcup\limits_{n \in \mu}E(p, m, j, s, n)$ follows obviously from the definitions of $U^\mu$ and $E(p, m, j, s, n).$ Let $$a \in  \bigcap_{p=1}^{\infty}\bigcap_{m=1}^{\infty}\bigcap_{j=1}^{\infty}\bigcap_{s=1}^{\infty}\bigcup_{n \in \mu}E(p, m, j, s, n).$$ 
		We shall show that $a \in U^\mu$. Let $K = \displaystyle\prod_{i = 1}^{d} K_i$ where $K_i \subseteq \mathbb{C}$ are compact with $\mathbb{C} \setminus K_i$ connected for all $i = 1, \dots d$ and there is at least one $i_0 \in \{1, \dots, d\}$ such that $0 \notin K_{i_0}$. Let also $h : \Omega \times K$ be a continuous function such that for every fixed $w \in \Omega$, the function $K \ni z \rightarrow h(w, z) \in \mathbb{C}$ belongs to $A_D(K)$ and for every fixed $z \in K$, the function $\Omega \ni w \rightarrow h(w, z) \in \mathbb{C}$ belongs to $H(\Omega)$.
		
		By the definition of the sets $T_m$, there exists a number $m_0 = 1, 2, \dots$, such that $K \subseteq T_{m_0}$. Now, we have that the restriction of $h$ to $F_1 \times K$ belongs to $A_D(F_1 \times K)$, so, by Lemma \ref{Merg}, there exists a polynomial $f_{j_1}, j_1 = 1, 2, \dots$ having coefficients whose coordinates are both rational, such that $$\sup_{(w,z) \in F_1 \times K}\big|h(w,z) - f_{j_1}(w,z)\big| < \frac{1}{2}.$$ 
		Let $p = 1, m = m_0, j = j_1$ and $s = 2$, then $a \in \bigcup_{n \in \mu}E(p, m, j, s, n)$. Thus, there exists $\lambda_1 \in \mu$ such that $$\sup_{(w,z) \in F_1 \times T_{m_0}}\big|S_{\lambda_1}(a)(w)(z) - f_{j_1}(w,z)\big| < \frac{1}{2}.$$
		As we have $\sup\limits_{(w,z) \in F_1 \times K}\big|h(w,z) - f_{j_1}(w,z)\big| < \dfrac{1}{2}$, $\sup\limits_{(w,z) \in F_1 \times T_{m_0}}\big|S_{\lambda_1}(a)(w)(z) - f_{j_1}(w,z)\big| < \dfrac{1}{2}$ and $K \subseteq T_{m_0}$, the triangle inequality implies $$\sup_{(w,z) \in F_1 \times K}\big|S_{\lambda_1}(a)(w)(z) - h(w,z)\big| < 1.$$
		Similarly, there exists $j_2 = 1, 2, \dots$, such that $$\sup_{(w,z) \in F_2 \times K}\big|h(w,z) - f_{j_2}(w,z)\big| < \frac{1}{4}.$$ 
		Let $p = 2, m = m_0, j = j_2$ and $s = 4$, then $a \in \bigcup_{n \in \mu}E(p, m, j, s, n)$. Thus, there exists $\lambda_2 \in \mu$ such that $$\sup_{(w,z) \in F_2 \times T_{m_0}}\big|S_{\lambda_2}(a)(w)(z) - f_{j_2}(w,z)\big| < \frac{1}{4}.$$
		and thus we have $$\sup_{(w,z) \in F_2 \times K}\big|S_{\lambda_2}(a)(w)(z) - h(w,z)\big| < \frac{1}{2}.$$
		Inductively, we can find a sequence of integers $(\lambda_p)_{p \in \mathbb{N}},\lambda_p\in \mu$ such that $$\sup_{(w,z) \in F_p \times K}\big|S_{\lambda_p}(a)(w)(z) - h(w,z)\big| < \frac{1}{p}$$
		for all $p = 1, 2, \dots$. Let $L$ be a compact subset of $\Omega$ and $\varepsilon > 0$. Since the family $F_p, p = 1, 2, \dots$ is an exhausting family of compact subsets of $\Omega$, there exists $p_1 = 1, 2, \dots$ such that $L \subseteq F_{p_1}$ and hence $L \subseteq F_p$ for each $p \geq p_1$ since the sequence $(F_p)_{p \in \mathbb{N}}$ is increasing. Let also $p_2 = 0, 1, \dots$ such that $\dfrac{1}{p_2} < \varepsilon$. Now, for $p \geq \max\{p_1, p_2\}$, we have that $$\sup_{(w,z) \in L \times K}\big|S_{\lambda_p}(a)(w)(z) - h(w,z)\big| \leq \sup_{(w,z) \in F_p \times K}\big|S_{\lambda_p}(a)(w)(z) - h(w,z)\big| < \frac{1}{p} < \varepsilon.$$
		Hence, $S_{\lambda_p}(a)(w)(z) \longrightarrow h(w,z)$ uniformly on $L \times K$ for every compact subset $L$ of $\Omega$. This proves that $a \in U^\mu$ and completes the proof.
	\end{proof}
	\begin{lemma}
	\label{Open}
		For every integer $p \geq 1, m \geq 1, j \geq 1, s \geq 1$ and $n \in \mu$, the set $E(p, m, j, s, n)$ is open in the space $H(\Omega)^{\aleph_0}$.
	\end{lemma}
	\begin{proof}
		Let $a = (a_0, a_1, \dots) \in E(p, m, j, s, n)$, then we have $$\sup_{(w,z) \in F_p \times T_m}\big|S_n(a)(w)(z) - f_j(w,z)\big| < \frac{1}{s}.$$ 
		Let $M := \max\big\{\sup_{z \in T_m}|z^{N_0}|, \dots, \sup_{z \in T_m}|z^{N_n}|\big\}$. We set now: $$\varepsilon = \dfrac{\frac{1}{s} - \sup_{(\xi,\eta) \in F_p \times T_m}\big|S_n(a)(\xi)(\eta) - f_j(\xi,\eta)\big|}{2(n + 1)M} > 0.$$
		For $k = 0, 1, \dots, n$ the function $b_k$ is continuous at $(a_0, \dots, a_k)$, so there exists $\delta_k > 0$ and a compact subset $L^{(k)}$ of $\Omega$ such that $$\sup\limits_{w \in F_p}|(b_k(c_0, c_1 \dots, c_k))(w) - (b_k(a_0, a_1 \dots, a_k))(w)| < \varepsilon$$ for $(c_0, c_1, \dots, c_k) \in H(\Omega)^{k + 1}$ with $\sup_{w \in L^{(k)}}|c_i(w) - a_i(w)| < \delta_k$ for $i = 0, 1, \dots, k$. We set $\delta = \min\{\delta_0, \delta_1, \dots, \delta_n\}$ and $L = \bigcup_{k = 0}^{n}L^{(k)}$. Suppose that $c = (c_0, c_1, \dots) \in H(\Omega)^{\aleph_0}$ satisfies $\sup_{w \in L}|c_k(w) - a_k(w)| < \delta$ for $k = 0, 1, \dots, n$. We shall show that $$\sup_{(w,z) \in F_p \times T_m}\big|S_n(c)(w)(z) - f_j(w,z)\big| < \frac{1}{s}$$ and therefore that $c \in E(p, m, j, s, n)$.
		This will prove that $E(p, m, j, s, n)$ is indeed open. For $(w,z) \in F_p \times T_m$, we have 
		\begin{equation*}
		\begin{split}
		& \big|S_n(c)(w)(z) - f_j(w,z)\big| \leq \big|S_n(c)(w)(z) - S_n(a)(w)(z)\big| + \big|S_n(a)(w)(z) - f_j(w,z)\big| \\
		& = \big|\sum_{k=0}^{n}(b_k(c_0, \dots , c_k))(w)z^{N_k} - \sum_{k=0}^{n}(b_k(a_0, \dots , a_k))(w)z^{N_k}\big| + \big|S_n(a)(w)(z) - f_j(w,z)\big| \\
		& \leq \sum_{k=0}^{n}|(b_k(c_0, \dots , c_k))(w) - (b_k(a_0, \dots , a_k))(w)|\cdot|z^{N_k}| + \big|S_n(a)(w)(z) - f_j(w,z)\big| \\
		& < \sum_{k=0}^{n}\varepsilon M + \big|S_n(a)(w)(z) - f_j(w,z)\big| \\
		& = \sum_{k=0}^{n}\dfrac{\frac{1}{s} - \sup_{(\xi,\eta) \in F_p \times T_m}\big|S_n(a)(\xi)(\eta) - f_j(\xi,\eta)\big|}{2(n + 1)} + \big|S_n(a)(w)(z) - f_j(w,z)\big| \\
		& = \frac{1}{2s} - \frac{1}{2}\sup_{(\xi,\eta) \in F_p \times T_m}\big|S_n(a)(\xi)(\eta) - f_j(\xi,\eta)\big| + \big|S_n(a)(w)(z) - f_j(w,z)\big|.
		\end{split}
		\end{equation*}
		Hence,
		\begin{equation*}
		\begin{split}
		& \sup_{(w,z) \in F_p \times T_m}\big|S_n(c)(w)(z) - f_j(w,z)\big| \\
		& \leq \frac{1}{2s} - \frac{1}{2}\sup_{(\xi,\eta) \in F_p \times T_m}\big|S_n(a)(\xi)(\eta) - f_j(\xi,\eta)\big| + \sup_{(w,z) \in F_p \times T_m}\big|S_n(a)(w)(z) - f_j(w,z)\big| \\
		& = \frac{1}{2s} + \frac{1}{2}\sup_{(w,z) \in F_p \times T_m}\big|S_n(a)(w)(z) - f_j(w,z)\big| < \frac{1}{2s} + \frac{1}{2s} = \frac{1}{s}.
		\end{split}
		\end{equation*}
		and the proof is completed.
	\end{proof}
	\begin{lemma}
		For every integer $p \geq 1, m \geq 1, j \geq 1$ and $s \geq 1$, the set $\bigcup\limits_{n \in \mu}E(p, m, j, s, n)$ is open and dense in the space $H(\Omega)^{\aleph_0}$.
	\end{lemma}
	\begin{proof}
		By Lemma \ref{Open} the sets $E(p, m, j, s, n), n \in \mu$ are open. Therefore the same is true for the union $\bigcup\limits_{n \in \mu}E(p, m, j, s, n)$. We shall prove that this set is also dense. Let $a = (a_0, a_1, \dots) \in H(\Omega)^{\aleph_0},$ $n_0$ be an integer such that $n_0 \geq 0$, $L$ be a compact subset of $\Omega$ and $\varepsilon > 0$. It suffices to find $n \in \mu$ and $c = (c_0, c_1, \dots) \in E(p, m, j, s, n)$, such that $$\sup\limits_{w \in L}|c_k(w) - a_k(w)| < \varepsilon \text{ for } k = 0, 1, \dots, n_0.$$
		The set $T_m$ is of the form $T_m = \displaystyle\prod_{i = 1}^{d} Q_i$, where one of the sets $Q_i$ does not contain 0, which we denote by $Q_{i_0}$. Let $N_k = (N_k^{(1)}, \dots, N_k^{(d)})$, $l = \max\{N_k^{(i_0)}, 0 \leq k \leq n_0, \}$, $l_0 = (0, \dots, l + 1, \dots, 0) \in \mathbb{N}^d$ where the number $l + 1$ is at the $i_0$-th position and $M := \sup\limits_{z_{i_0} \in Q_{i_0}}|z_{i_0}^{l + 1}|$. We notice that $z_{i_0}^{l + 1} = z^{l_0}$. We set $c_k = a_k$ for $k = 0, 1, \dots, n_0$ and so $ b_k(c_0, \dots, c_k) = b_k(a_0, \dots, a_k) \text{ for } k = 0, 1, \dots, n_0.$. We need to find $n \in \mu$ such that $$\sup_{(w,z) \in F_p \times T_m}\big|S_n(c)(w)(z) - f_j(w,z)\big| < \frac{1}{s}.$$
		We have 
		\begin{equation*}
		\begin{split}
		& \sup\limits_{(w,z) \in F_p \times T_m}\big|S_n(c)(w)(z) - f_j(w,z)\big| = \sup\limits_{(w,z) \in F_p \times T_m}\big|\sum_{k=0}^{n}(b_k(c_0, \dots, c_k))(w)z^{N_k} - f_j(w,z)\big| \\
		& = \sup\limits_{(w,z) \in F_p \times T_m}\big|\sum_{k = n_0 + 1}^{n}(b_k(c_0, \dots, c_k))(w)z^{N_k} + \sum_{k=0}^{n_0}(b_k(a_0, \dots, a_k))(w)z^{N_k} - f_j(w,z)\big| \\
		& = \sup\limits_{(w,z) \in F_p \times T_m}\big|z_{i_0}^{l + 1}\sum_{k = n_0 + 1}^{n}(b_k(c_0, \dots, c_k))z^{N_k - l_0} + \sum_{k=0}^{n_0}(b_k(a_0, \dots, a_k))(w)z^{N_k} - f_j(w,z)\big| \\
		& = \sup\limits_{(w,z) \in F_p \times T_m}|z_{i_0}^{l + 1}|\big|\sum_{k = n_0 + 1}^{n}(b_k(c_0, .., c_k))(w)z^{N_k - l_0} - \frac{f_j(w,z) - \sum_{k=0}^{n_0}(b_k(a_0, .., a_k))(w)z^{N_k}}{z_{i_0}^{l + 1}}\big| \\
		& \leq M\sup\limits_{(w,z) \in F_p \times T_m}\big|\sum_{k = n_0 + 1}^{n}(b_k(c_0, .., c_k))(w)z^{N_k - l_0} - \frac{f_j(w,z) - \sum_{k=0}^{n_0}(b_k(a_0, .., a_k))(w)z^{N_k}}{z_{i_0}^{l + 1}}\big|.
		\end{split}
		\end{equation*}
		By Lemma \ref{Merg}, there exists a polynomial $p(w,z)$ such that $$\sup\limits_{(w,z) \in F_p \times T_m}\big|p(w,z) - \frac{f_j(w,z) - \sum_{k=0}^{n_0}(b_k(a_0, \dots, a_k))(w)z^{N_k}}{z_{i_0}^{l + 1}}\big| < \frac{1}{Ms}.$$
		The polynomial $p(w,z)$ can be written in the form $p(w,z) = \sum\limits_{k = n_0 + 1}^{\infty}p_k(w)z^{N_k - l_0}$ where $p_k$ are polynomials, such that all but finitely many of them are identically equal to 0. The image of the function $H(\Omega)\ni f \rightarrow b_{n_0 + 1}(a_0, a_1, \dots , a_{n_0}, f)$ contains the polynomials so we can find $c_{n_0 + 1} \in H(\Omega)$ such that $b_{n_0 + 1}(a_0, \dots, a_{n_0}, c_{n_0 + 1}) = p_{n_0 + 1}$. Similarly, we can find $c_{n_0 + 2}, c_{n_0 + 3}, \dots \in H(\Omega)$ such that $b_k(a_0, \dots, a_{n_0}, c_{n_0 + 1}, \dots c_k) = p_k$ for all $k = n_0 + 1, n_0 + 2, \dots$ and so $\sum\limits_{k = n_0 + 1}^{\infty}(b_k(c_0, \dots, c_k))(w)z^{N_k - l_0} = \sum\limits_{k = n_0 + 1}^{\infty}p_k(w)z^{N_k - l_0} = p(w, z)$. Since all but finitely many of the function $p_k$ are identically equal to 0, there exists an integer $n'$ such that $p_n \equiv 0$ for all $n \geq n'$. By choosing $n \in \mu$ such that $n \geq n'$ we have $$\sup\limits_{(w,z) \in F_p \times T_m}\big|S_n(c)(w)(z) - f_j(w,z)\big| < \frac{1}{s}.$$
		This proves that the set $\bigcup\limits_{n \in \mu}E(p, m, j, s, n)$ is indeed dense.
	\end{proof}
	\begin{theorem}
	\label{Main}
		Under the above assumptions and notation, the set $U^\mu$ is a $G_\delta$ and dense subset of the space $H(\Omega)^{\aleph_0}$.
	\end{theorem}
	\begin{proof}
		The result is obvious by combining the previous lemmas with Baire's Theorem.
	\end{proof}
	\begin{theorem}
		Under the above assumptions and notation, assuming in addition that the functions $b_k$ are linear, then the set $U^\mu \cup \{0\}$ contains a vector space, dense in $H(\Omega)^{\aleph_0}$.
	\end{theorem}
	The proof uses the result of Theorem \ref{Main}, follows the lines of the implication $(3) \implies (4)$ of the proof of Theorem 3 in \cite{Bayart} and is omitted.
	\begin{remark}
		We have proven that, for a fixed enumeration of $\mathbb{N}^d$, the set of formal power series in $H(\Omega)^{\aleph_0}$ having the desired universal approximation property with respect to this enumeration is a $G_\delta$ and dense set. Using Baire's Theorem, it easily follows that the set of formal power series in $H(\Omega)^{\aleph_0}$ that have the same universal approximation property with respect to any countable family of enumerations of $\mathbb{N}^d$ is a $G_\delta$ and dense set. Since the set of all enumerations of $\mathbb{N}^d$ is uncountable, a natural question that arises is whether we can generalise the result when the formal power series have the universal approximation property with respect to all enumerations of $\mathbb{N}^d$. The answer to this question is negative. The proof is similar to a result in Section 6 of \cite{Kioulafa}.
	\end{remark}
	\section{Universal power series without parameters}
	\begin{definition}
		For every integer $k\geq0$ let $b_k : \mathbb{C}^{k + 1}\rightarrow \mathbb{C}$ be a continuous function such that for every $k\geq1$ and $a_0, a_1, \dots , a_{k - 1} \in \mathbb{C}$, the function $$\mathbb{C}\ni z \rightarrow b_k(a_0, a_1, \dots , a_{k - 1}, z)$$ is onto $\mathbb{C}$. We also assume that the function $\mathbb{C}\ni z \rightarrow b_0(z)$ is onto $\mathbb{C}$.
		
		Let $N_k, k = 0, 1, \dots$ be an enumeration of $\mathbb{N}^d$ and $a = (a_0, a_1, \dots) \in \mathbb{C}^{\aleph_0}$. For every integer $n \geq 0$ and $z \in \mathbb{C}^d$ we set $S_n(a)(z) = \sum\limits_{k = 0}^{n}(b_k(a_0, \dots , a_k))z^{N_k}$. Let $\mu$ be an infinite subset of $\mathbb{N}$. We define $U^\mu$ to be the set of $a \in \mathbb{C}^{\aleph_0}$ such that for every set $K = \displaystyle\prod_{i = 1}^{d} K_i$ where $K_i \subseteq \mathbb{C}$ are compact with $\mathbb{C} \setminus K_i$ connected for all $i = 1, \dots d$ and there is at least one $i_0 \in \{1, \dots, d\}$ such that $0 \notin K_{i_0}$ and every function $h \in A_D(K)$, there exists a strictly increasing sequence $\lambda_n \in \mu, n = 1, 2, \dots$ such that $$S_{\lambda_n}(a)(z) \longrightarrow h(z)\text{, uniformly on }  K.$$
	\end{definition}
	\begin{theorem}
		\label{Main_2}
		Under the above assumptions, the set $U^\mu$ is a $G_\delta$ and dense subset of $\mathbb{C}^{\aleph_0}$ endowed with the Cartesian topology. If in addition the functions $b_k$ are linear, then $U^\mu$ contains a dense vector space except 0.
	\end{theorem}
	We notice that if we assume that there exists a sequence of integers $(\lambda_n)_{n \in \mathbb{N}},\lambda_n\in \mu$, not necessarily strictly increasing, such that $(S_{\lambda_n}(a)(z))_{n \in \mathbb{N}} \longrightarrow h(w,z)$ uniformly on $K$, then the two definitions are equivalent; see \cite{Vlachou}.
	
	Let $T_m, m = 1, 2, \dots$ be as in the previous section and $f_j, j = 1, 2, \dots$ be an enumeration of all polynomials of $d$ variables having coefficients with rational coordinates. For any integers $m, j, s, n$ with $m \geq 1, j \geq 1, s \geq 1, n \geq 0$, we denote by $E(m, j, s, n)$ the set $$E(m, j, s, n) := \Big\{a \in \mathbb{C}^{\aleph_0} : \sup_{z \in T_m}\big|S_n(a)(z) - f_j(z)\big| < \frac{1}{s}\Big\}.$$
	We will prove some useful lemmas.
	\begin{lemma}
		$U^\mu$ can be written as follows:
		$$U^\mu = \bigcap_{m=1}^{\infty}\bigcap_{j=1}^{\infty}\bigcap_{s=1}^{\infty}\bigcup_{n \in \mu}E(m,j,s,n)$$
	\end{lemma}
	\begin{proof}
		The fact that $U^\mu \subseteq \bigcap_{m=1}^{\infty}\bigcap_{j=1}^{\infty}\bigcap_{s=1}^{\infty}\bigcup_{n \in \mu}E(m,j,s,n)$ follows easily from the definition.\\
		Let $a \in  \bigcap_{m=1}^{\infty}\bigcap_{j=1}^{\infty}\bigcap_{s=1}^{\infty}\bigcup_{n \in \mu}E(m,j,s,n)$, $K = \prod\limits_{i=1}^dK_i$ and $h \in A_D(K)$ as described in the definition. We know that there exists a number $m_0 = 1, 2, \dots$ such that $K \subseteq T_{m_0}$.
		Let $\varepsilon > 0$. By Lemma \mbox{\ref{Merg}}, there exists a polynomial $f_{j_0}$ such that $\sup\limits_{z \in K} |f_{j_0}(z)-h(z)|< \frac{\varepsilon}{2}$.
		Let also $s_0 = 1, 2, \dots$ such that $\frac{1}{s_0}<\frac{\varepsilon}{2}$.
		Since $a \in \bigcup_{n \in \mu}E(m_0,j_0,s_0,n)$, there exists a $n_0 \in \mu$ such that $\sup\limits_{z \in T_{m_0}}|f_j(z)-S_{n_0}(a)(z)| < \frac{1}{s_0}<\frac{\varepsilon}{2}$.
		Using the triangle inequality we conclude that $\sup\limits_{z \in K}|S_{n_0}(a)(z)-h(z)|<\varepsilon$.
	\end{proof}
	\begin{lemma}
		For every integer $m \geq 1, j \geq 1, s \geq 1$ and $n\in\mu$, the set $E(m,j,s,n)$ is open in $\mathbb{C}^{\aleph_0}$.
	\end{lemma}
	\begin{proof}
		Let $a = (a_0, a_1, \dots) \in E(m,j,s,n)$.
		Then we have $\sup\limits_{z \in T_m}|S_n(a)(z) - f_j(z)| < \frac{1}{s}$.
		Let $M= \max\limits_{0 \leq k \leq n}\big\{ \sup\limits_{z\in T_m}|z^{N_k}| \big\}$
		We set now:
		$\varepsilon = \frac{\frac{1}{s} - \sup\limits_{z \in T_m}|S_n(a)(z) - f_j(z)|}{2(n+1)M}$.
		Since for every $k \leq n$, the function $b_k: \mathbb{C}^{k+1} \rightarrow \mathbb{C}$ is continuous, there exists $\delta_k > 0$ such that $|b_k(a_0, \dots, a_k) - b_k(c_0, \dots, c_k)| < \varepsilon$ for all $(c_0, \dots, c_k) \in D(a_0, \delta_k) \times \dots \times D(a_k, \delta_k)$.
		Let $\delta = \min\limits_{0 \leq k \leq n} \{ \delta_k \}$, and $c = (c_1, c_2, \dots) \in \mathbb{C}^{\aleph_0}$ with $c_k \in D(a_k, \delta) $ for all $0 \leq\ k \leq n$.
		For every $z \in T_m$ we have:
		\begin{equation*}
		\begin{split}
		&|S_n(a)(z) - S_n(c)(z)| = \big| \sum\limits_{k=0}^n (b_k(a_0, \dots, a_k) - b_k(c_0, \dots, c_k))z^{N_k}\big| \leq\\
		&\leq M\big|\sum\limits_{k=0}^n (b_k(a_0, \dots, a_k) - b_k(c_0, \dots, c_k))\big| \leq \\
		&\leq M(n+1)\varepsilon = \frac{1}{2s} - \frac{1}{2}\sup\limits_{z\ \in T_m}|S_n(a)(z) - f_j(z)|.
		\end{split}
		\end{equation*}
		And thus 
		\begin{align*}
		| S_n(c)(z) - f_j(z)| \leq |S_n(a)(z) &- S_n(c)(z)| + |S_n(a)(z) - f_j(z)| < \frac{1}{s}
		\end{align*}
		and so $c \in E(m, j, s, n)$.
	\end{proof}
	\begin{lemma}
		For every integer $m \geq 1, j \geq 1, s \geq 1$, the set $\bigcup\limits_{n\in\mu}E(m,j,s,n)$ is open and dense in $\mathbb{C}^{\aleph_0}$.
	\end{lemma}
	\begin{proof}
		The sets $E(m,j,s,n), n \in \mu$ are open. Therefore the same is true for their union. We shall prove that it is also dense.
		Let $a = (a_0, a_1, \dots ) \in \mathbb{C}^{\aleph_0}, n_0 \in \mathbb{Z}$ such that $n_0 \geq 0$, and $\varepsilon > 0$.
		It suffices to find $n \in \mu$ and $c = (c_0, c_1 \dots) \in E(m,j,s,n)$ such that  $c_k \in D(a_k, \varepsilon)$ for $0\leq k \leq n_0$.
		
		The set $T_m$ is of the form $T_m = \prod\limits_{i=1}^dQ_i$ where one of the sets $Q_i$ does not contain $0$, which we denote by $Q_{i_0}$.
		Let $l=\max\{ N_k^{(i_0)} | 0\leq k \leq n_0\}$, $l_0 = (0, \dots, l+1, \dots , 0)$ where the number $l+1$ is at the $i_0$-th position and M = $\max\limits_{z_{i_0}\in Q_{i_0}}|z_{i_0}^{l+1}|$.
		
		We set $c_k = a_k$ for $0 \leq k \leq n_0 $ which implies $ b_k(a_0, \dots, a_k) = b_k(c_0, \dots, c_k)$ for $0 \leq k \leq n_0$.
		We have:\\
		\begin{align*}
		&\sup\limits_{z\in T_m}{|S_n(c)(z) - f_j(z)|} = \\
		&= \sup\limits_{z\in T_m}{\big|\sum\limits_{k=n_0+1}^{n}b_k(c_0, \dots, c_k)z^{N_k} + \sum\limits_{k=0}^{n_0}b_k(a_0, \dots, a_k)z^{N_k} - f_j(z)\big|} =\\
		&= \sup\limits_{z\in T_m}{ |z_{i_0}^{l+1}| \cdot \big|\sum_{k=n_0+1}^{n}b_k(c_0, \dots, c_k)z^{N_k-l_0} - \frac{f_j(z) - \sum_{k=0}^{n_0}b_k(a_0, \dots, a_k)z^{N_k}}{z_{i_0}^{l+1}} \big| } \leq \\
		&\leq M \sup_{z\in T_m}{ \big|\sum_{k=n_0+1}^{n}b_k(c_0, \dots, c_k)z^{N_k-l_0} - \frac{f_j(z) - \sum_{k=0}^{n_0}b_k(a_0, \dots, a_k)z^{N_k}}{z_{i_0}^{l+1}} \big| }.
		\end{align*}
		By Lemma \mbox{\ref{Merg}} there exists a polynomial $p(z) = \sum\limits_{k=n_0+1}^{n'}p_kz^{N_k-l_0}$ such that:
		$$\sup\limits_{z\in T_m}|p(z) - \frac{f_j(z) - \sum_{k=0}^{n_0}b_k(a_0, \dots, a_k)z^{N_k}}{z_{i_0}^{l+1}} | < \frac{1}{Ms}$$
		Let $n \in \mu: n \geq n'$.
		From the surjectivity of $\mathbb{C} \ni z \rightarrow b_k(c_0, \dots, c_{k-1}, z)$, there exist $c_{n_0+1}, c_{n_0 + 2}, \dots \in \mathbb{C}$ such that $b_k(c_0, \dots, c_k) = p_k$ for $n_0+1 \leq k \leq n'$ and $b_k(c_0, \dots, c_k) = 0$ for $k > n'$.
		With this choice of $c_k$ we have $\sum\limits_{k=n_0+1}^{n}b_k(c_0, \dots, c_k)z^{N_k-l_0} = p(z)$ and thus:
		$$\sup\limits_{z\in T_m}{|S_n(c)(z) - f_j(z)|} < \frac{1}{s}$$
	\end{proof}
	\begin{proof}[Proof of Theorem \ref{Main_2}]
		The space $\mathbb{C}^{\aleph_0}$ endowed with the Cartesian topology is a complete metric space. So, by combining the previous lemmas with Baire's Theorem, we get that $U^\mu$ is a $G_\delta$ and dense set. If we also assume that the functions $b_k$ are linear, we can prove that $U^\mu$ contains a dense vector space except 0. The proof uses the previous result, follows the lines of the implication $(3) \implies (4)$ of the proof of Theorem 3 in \cite{Bayart} and is omitted.
	\end{proof}
	\vspace{10mm}
	\par
	\textit{Acknowledgement}. The authors would like to thank Professor Vassili Nestoridis for helpful communications and guidance throughout the creation of this paper.
	
	\vspace{10mm}
	\itshape
	Department of Mathematics\\
	Panepistimiopolis\\
	National and Kapodistrian University of Athens\\
	Athens, 15784\\
	Greece\\[\baselineskip]
	E-mail Addresses:\\
	conmaron@gmail.com\\
	giwrg98@gmail.com
\end{document}